\documentclass{amsart}

\usepackage{etex}
\usepackage{amsmath}
\usepackage{amsthm}
\usepackage{amssymb}
\usepackage[frame,cmtip,arrow,matrix,line,graph,curve]{xy}
\usepackage{graphpap, color, paralist, pstricks}
\usepackage[mathscr]{eucal}
\usepackage{mathabx}
\usepackage[pdftex,colorlinks,backref=page,citecolor=blue]{hyperref}
\usepackage{tikz}
\usetikzlibrary{calc,graphs,graphs.standard}
\usepackage{graphicx}
\usepackage{subcaption}
\usepackage{verbatim}

\usepackage{amsmath}
\usepackage{amsthm,amsfonts,amssymb,mathrsfs}
\usepackage{epic,eepic}
\usepackage{yfonts}
\usepackage{paralist,enumerate}

\theoremstyle{definition}
\newtheorem{theorem}{Theorem}
\newtheorem{definition}[theorem]{Definition}
\newtheorem{conjecture}[theorem]{Conjecture}
\newtheorem{lemma}[theorem]{Lemma}
\newtheorem{proposition}[theorem]{Proposition}
\newtheorem{corollary}[theorem]{Corollary}

\newtheorem*{theorem*}{Theorem}

\theoremstyle{remark}

\newtheorem{example}[theorem]{Example}

\begin{document}

\title{Correlation bounds for fields and matroids}

\author{June Huh, Benjamin Schr\"oter, and Botong Wang}

\address{Institute for Advanced Study, Princeton, NJ, USA.}
\email{junehuh@ias.edu}

\address{Institut f\"ur Mathematik, TU Berlin, Berlin, Germany.}
\email{schroeter@math.tu-berlin.de}

\address{University of Wisconsin-Madison, Madison, WI, USA.}
\email{wang@math.wisc.edu}

\maketitle

\section{Introduction and  results}

Let $\mathrm{G}$ be a finite connected graph, and let $w=(w_e)$ be a set of positive weights on the  edges $e$ of $\mathrm{G}$.
Randomly pick a spanning tree $\mathrm{T}$ of $\mathrm{G}$ so that the probability of selecting an individual tree $t$ is proportional to  the product of the weights of its edges:
\[
\mathbb{P}(\mathrm{T}=t)  \ \ \propto \ \ \prod_{e \in t} w_e.
\]
The work of Kirchhoff on electrical networks can be used to show that, for any distinct edges $i$ and $j$, the events $i \in \mathrm{T}$ and $j \in \mathrm{T}$ are negatively correlated:
\[
\mathbb{P}(\text{$\mathrm{T}$ contains $i$} \mid \text{$\mathrm{T}$ contains $j$}) \le \mathbb{P}(\text{$\mathrm{T}$ contains $i$}).
\]
Equivalently, for any distinct edges $i$ and $j$, we have
\[
 \mathbb{P}(i \in \mathrm{T}, j \in \mathrm{T})\ \mathbb{P}(i \notin \mathrm{T}, j \notin \mathrm{T}) \le 
 \mathbb{P}(i \in \mathrm{T}, j \notin \mathrm{T})\ \mathbb{P}(i \notin \mathrm{T}, j \in \mathrm{T}).
\]
We refer to \cite{Pemantle95} and \cite[Chapter 4]{LP} for modern expositions.

Let $E$ be a finite set.
A \emph{matroid} on $E$ is a nonempty collection of subsets of $E$, called \emph{bases} of the matroid, that satisfies the exchange property:
\[
\text{For any bases $b_1,b_2$ and $e_1 \in b_1 \setminus b_2$, there is $e_2 \in b_2 \setminus b_1$ such that $\big( b_1 \setminus e_1 \big)\cup e_2$ is a basis.}
\]
An \emph{independent set} is a subset of a basis, a \emph{dependent set} is a subset of $E$ that is not independent,
a \emph{circuit} is a minimal dependent set, the \emph{rank} of a subset of $E$ is the cardinality of any one of its maximal independent subsets,
 and a \emph{flat} is a subset of $E$ that is maximal for its rank.
The \emph{rank} of a matroid is the cardinality of any one of its bases.
For any unexplained matroid terms and facts, we refer  to Oxley's book \cite{Oxley}.
The collection of spanning trees of a connected graph is the best-known example of a matroid.

Let $\mathrm{M}$ be a matroid on  $E$, and fix a set of positive weights $w=(w_e)$ on the elements $e$ of $E$.
Randomly pick a basis $\mathrm{B}$ of the matroid so that the probability of selecting an individual basis $b$ is proportional to  the product of the weights of its elements:
\[
\mathbb{P}(\mathrm{B}=b)  \ \ \propto \ \ \prod_{e \in b} w_e.
\]
In this more general setup, for any distinct  $i$ and $j$ in $E$, do we still have the negative correlation
\[
 \mathbb{P}(i \in \mathrm{B}, j \in \mathrm{B})\ \mathbb{P}(i \notin \mathrm{B}, j \notin \mathrm{B}) \le 
 \mathbb{P}(i \in \mathrm{B}, j \notin \mathrm{B})\ \mathbb{P}(i \notin \mathrm{B}, j \in \mathrm{B})?
\]
The answer is ``yes'' if the matroid is regular \cite{FM}, if the matroid is representable over $\mathbb{F}_3$ and $\mathbb{F}_4$ \cite{COSW}, if the cardinality of $E$ is at most $7$, or if the rank of $\mathrm{M}$ is at most $3$ \cite{Wagner05}.
Examples below show that distinct elements of $E$ can define positively correlated events for  more general matroids.

\begin{figure}[t]
\begin{subfigure}[b]{0.3\linewidth}
\centering
    \begin{tikzpicture}[scale = 0.35, color = {black}]
      \tikzstyle{linestyle} = [thick, line cap=round, line join=round];

      \coordinate  (v1) at (0:3);
      \coordinate  (v2) at (60:3);
      \coordinate  (v3) at (120:3);
      \coordinate  (v4) at (180:3);
      \coordinate  (v5) at (240:3);
      \coordinate  (v6) at (300:3);
      \coordinate  (a1) at (120:1.3);
      \coordinate  (a2) at (240:1.3);
      \coordinate  (a3) at (0:1.3);
      
      \draw [black, fill=lightgray] (a1) -- (a2) -- (a3) -- cycle;
      \node [fill=lightgray, inner sep = 1pt, circle] at (0,0) {\small $i$};
      
      \draw[linestyle] (v1) -- (v2) -- (v3) -- (v4) -- (v5) -- (v6) -- cycle;
      \draw[linestyle] (a1) -- (a2) -- (a3) --  cycle;
      \draw[linestyle] (v1) -- (a3) -- (v2)  -- (a1) -- (v3);
      \draw[linestyle] (a1) --(v4) -- (a2) -- (v5);
      \draw[linestyle] (a2) --(v6) -- (a3);

       \draw[fill=black] (v1) circle (3pt);
       \draw[fill=black] (v2) circle (3pt);
       \draw[fill=black] (v3) circle (3pt);
       \draw[fill=black] (v4) circle (3pt);
       \draw[fill=black] (v5) circle (3pt);
       \draw[fill=black] (v6) circle (3pt);   
       \draw[fill=black] (a1) circle (3pt);
       \draw[fill=black] (a2) circle (3pt);
       \draw[fill=black] (a3) circle (3pt);

       \node[right] at (v1) {\small $4$};               
       \node[above right] at (v2) {\small $6$};
       \node[above left] at (v3) {\small $5$};               
       \node[left] at (v4) {\small $4$};   
       \node[below left] at (v5) {\small $6$};               
       \node[below right] at (v6) {\small $5$};                 
       \node[below right] at (a3) {\small $1$};     
       \node[left] at (125:1.7) {\small $3$};
       \node[left] at (235:1.7) {\small $2$};
        
      \coordinate  (t) at (345:6.9);
      \coordinate  (w1) at ($(0:3)+(t)$);
      \coordinate  (w2) at ($(60:3)+(t)$);
      \coordinate  (w3) at ($(120:3)+(t)$);
      \coordinate  (w4) at ($(180:3)+(t)$);
      \coordinate  (w5) at ($(240:3)+(t)$);
      \coordinate  (w6) at ($(300:3)+(t)$);
      \coordinate  (b1) at ($(60:1.3)+(t)$);
      \coordinate  (b2) at ($(180:1.3)+(t)$);
      \coordinate  (b3) at ($(300:1.3)+(t)$);

      \draw [black, fill=lightgray] (b1) -- (b2) -- (b3) -- cycle;
       \node [fill=lightgray, inner sep = 1pt, circle] at (t) {\small $j$};
     
      \draw[linestyle] (w1) -- (w2) -- (w3) -- (w4) -- (w5) -- (w6) --cycle ;
      \draw[linestyle] (b1) -- (b2) -- (b3)  --cycle ; 
      \draw[linestyle] (w6) -- (b3) -- (w1) -- (b1) -- (w2) ;
      \draw[linestyle] (b3) -- (w5) -- (b2) -- (w4);
      \draw[linestyle] (b2) -- (w3) -- (b1);

       \draw[fill=black] (w1) circle (3pt);
       \draw[fill=black] (w2) circle (3pt);
       \draw[fill=black] (w3) circle (3pt);
       \draw[fill=black] (w4) circle (3pt);
       \draw[fill=black] (w5) circle (3pt);
       \draw[fill=black] (w6) circle (3pt);   
       \draw[fill=black] (b1) circle (3pt);
       \draw[fill=black] (b2) circle (3pt);
       \draw[fill=black] (b3) circle (3pt);

       \node[right] at (w1) {\small $1$};               
       \node[above right] at (w2) {\small $3$};
       \node[above left] at (w3) {\small $2$};               
       \node[left] at (w4) {\small $1$};   
       \node[below left] at (w5) {\small $3$};               
       \node[below right] at (w6) {\small $2$};                 
       \node[below left] at (b2) {\small $6$};                 
       \node at ($(315:1.9)+(t)$) {\small $4$};
       \node at ($(45:1.9)+(t)$) {\small $5$};
                        
    \end{tikzpicture}
     
   \caption{simplicial matroid}\label{fig:simplicial}
\end{subfigure}
\hfill
\begin{subfigure}[b]{0.3\linewidth}
  \centering
     \begin{tikzpicture}[scale = 0.4, color = {black}]
      \tikzstyle{linestyle} = [thick, line cap=round, line join=round];
      \coordinate  (u) at (0,0.85*3);
      \coordinate  (d) at (0,-0.85*3);
      \coordinate  (v1) at (-0.85*5,0);
      \coordinate  (v2) at (-0.85*3.2,0);
      \coordinate  (v3) at (-0.85*1.5,0);
      \coordinate  (v4) at (0.85*1.5,0);
      \coordinate  (v5) at (0.85*3.2,0);
      \coordinate  (v6) at (0.85*5,0);

      \draw[linestyle] (u) to [bend right] (v1) to [bend right] (d) -- cycle;
      \draw[linestyle] (u) to [bend right] (v2) to [bend right]  (d) ;
      \draw[linestyle] (u) to [bend right] (v3) to [bend right]  (d) ;
      \draw[linestyle] (u) to [bend left] (v4) to [bend left]  (d) ;
      \draw[linestyle] (u) to [bend left] (v5) to [bend left]  (d) ;
      \draw[linestyle] (u) to [bend left] (v6) ;


      \node[right] at ($0.5*(u)+0.5*(d)$) {$i$};
      \node[above right] at ($0.6*(u)+0.6*(v6)$) {$j$};

       \draw[fill=black] (v1) circle (4pt);
       \draw[fill=black] (v2) circle (4pt);
       \draw[fill=black] (v3) circle (4pt);
       \draw[fill=black] (v4) circle (4pt);
       \draw[fill=black] (v5) circle (4pt);
       \draw[fill=black] (v6) circle (4pt);

       \draw[fill=black] (u) circle (4pt);
       \draw[fill=black] (d) circle (4pt);

       \draw (-6,-3.5) circle (0pt);

   \end{tikzpicture}
\caption{truncated graphic matroid}\label{fig:graph}
\end{subfigure}
\hfill
\begin{subfigure}[b]{0.3\linewidth}
  \centering
   \begin{tikzpicture}[scale=0.4,color=black]

      \coordinate  (c1) at (0,1.73-0.8);
      \coordinate  (c2) at (1,-0.8);
      \coordinate  (c3) at (-1,-0.8);
      \coordinate  (j) at ($0.33*(c1)+0.33*(c2)+0.33*(c3)+(0,0.4)$);
      \coordinate  (c4) at ($0.33*(c1)+0.33*(c2)+0.33*(c3)$);      
      \coordinate  (i) at (-3,1.4);
      \coordinate  (a1) at (-0.6,1.8);
      \coordinate  (a2) at (-0.2,1.8);
      \coordinate  (a3) at (0.2,1.8);
      \coordinate  (a4) at (0.6,1.8);
      \coordinate  (b1) at ($(c2)+0.5*(0.282,0.282)+(0.7,-0.55)$);
      \coordinate  (b2) at ($(b1)+1.0*(0.282,0.282)$);
      \coordinate  (b3) at ($(c2)-0.5*(0.282,0.282)+(0.7,-0.55)$);
      \coordinate  (d1) at ($(c3)+0.5*(-0.282,0.282)+(-0.7,-0.55)$);
      \coordinate  (d2) at ($(d1)+1.0*(-0.282,0.282)$);
      \coordinate  (d3) at ($(c3)-0.5*(-0.282,0.282)+(-0.7,-0.55)$);
                
	\draw (c1) circle (2.1) ;
	\draw (c2) circle (2.1) ;
	\draw (c3) circle (2.1) ;
	\draw (c4) circle(4) ;

       \draw[fill=black] (i) circle (3pt);
       \draw[fill=black] (j) circle (3pt);

       \draw[fill=black] (a1) circle (3pt);
       \draw[fill=black] (a2) circle (3pt);
       \draw[fill=black] (a3) circle (3pt);
       \draw[fill=black] (a4) circle (3pt);

       \draw[fill=black] (b1) circle (3pt);
       \draw[fill=black] (b2) circle (3pt);
       \draw[fill=black] (b3) circle (3pt);

       \draw[fill=black] (d1) circle (3pt);
       \draw[fill=black] (d2) circle (3pt);
       \draw[fill=black] (d3) circle (3pt);

       \node[below] at (i) {$i$};              
       \node[below] at (j) {$j$};
       \node at (0,2.5) {\small $A_2$};
       \node at (-1.3,-2.3) {\small $A_3$};
       \node at (1.3,-2.3) {\small $A_4$};
       \node at (0,-3.7) {\small $A_1$};

    \end{tikzpicture}
    \caption{transversal matroid}\label{fig:diagram}
\end{subfigure}
\caption{Positive correlation in matroids}
\end{figure}
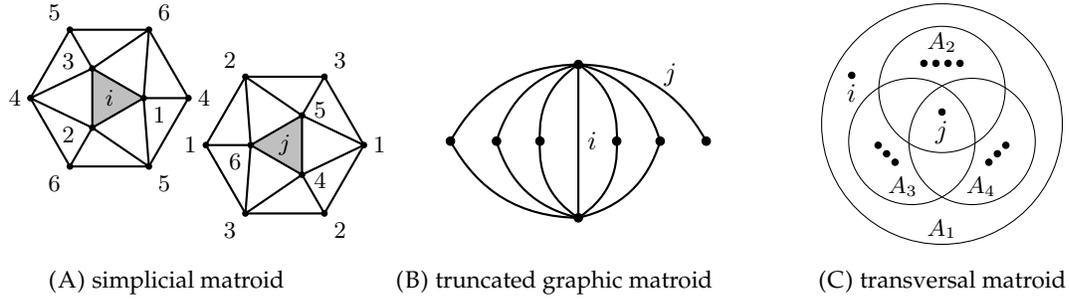

\begin{example}
Let $\mathrm{S}$ be the $2$-dimensional skeleton of the $5$-dimensional simplex.
A spanning tree of $\mathrm{S}$ is a maximal subset of the twenty triangles in $\mathrm{S}$ that does not contain any $2$-cycle over $\mathbb{F}_2$.
Choose one such $\mathrm{B}$ uniformly at random.
Then, for any two disjoint triangles in $\mathrm{S}$, say $i=123$ and $j=456$ in Figure \ref{fig:simplicial},
we have
\begin{align*}
\mathbb{P}(i \in \mathrm{B}, j \in \mathrm{B}) \ \mathbb{P}(i \notin \mathrm{B}, j \notin \mathrm{B})&=\frac{11664}{46608}\cdot \frac{11664}{46608} \simeq 0.06263,\\
\mathbb{P}(i \in \mathrm{B}, j \notin \mathrm{B})\ \mathbb{P}(i \notin \mathrm{B}, j \in \mathrm{B})&=\frac{11640}{46608}\cdot \frac{11640}{46608} \simeq 0.06237.
\end{align*}
This example was found by Andrew Newman.
\end{example}

\begin{example}\label{ex:graphic}
Let $\mathrm{G}$ be the graph in Figure \ref{fig:graph}.
Consider the collection of all forests in $\mathrm{G}$ with exactly six edges, and choose one such $\mathrm{B}$ uniformly at random.
Then, for the edges labelled $i$ and $j$ in Figure \ref{fig:graph}, we have
\begin{align*}
\mathbb{P}(i \in \mathrm{B}, j \in \mathrm{B}) \ \mathbb{P}(i \notin \mathrm{B}, j \notin \mathrm{B})&=\frac{80}{384}\cdot \frac{80}{384} \simeq 0.04340,\\
\mathbb{P}(i \in \mathrm{B}, j \notin \mathrm{B})\ \mathbb{P}(i \notin \mathrm{B}, j \in \mathrm{B})&=\frac{32}{384}\cdot \frac{192}{384} \simeq 0.04167.
\end{align*}
This example, attributed to Paul Seymour, Peter Winkler, and Madhu Sudan, is discussed in \cite[Section 2]{FM}.
\end{example}

\begin{example}\label{ex:transversal}
Let $A_1,A_2,A_3,A_4$ be the four finite sets shown in Figure \ref{fig:diagram}.
A system of distinct representatives 
is a set $\{x_1,x_2,x_3,x_4\}$ of size four such that $x_k \in A_k$ for all $k$.
Choose one such $\mathrm{B}$ uniformly at random.
Then, for the elements labelled $i$ and $j$ in Figure \ref{fig:diagram},
we have
\begin{align*}
\mathbb{P}(i \in \mathrm{B}, j \in \mathrm{B}) \ \mathbb{P}(i \notin \mathrm{B}, j \notin \mathrm{B})&=\frac{33}{309}\cdot \frac{126}{309} \simeq 0.04355,\\
\mathbb{P}(i \in \mathrm{B}, j \notin \mathrm{B})\ \mathbb{P}(i \notin \mathrm{B}, j \in \mathrm{B})&=\frac{36}{309}\cdot \frac{114}{309} \simeq 0.04298.
\end{align*}
This example is from \cite[Section 5]{CW}.\footnote{The proof of \cite[Proposition 5.9]{CW} needs a small correction. 
In the notation of that paper, the numbers should be
$L_e=69$, $L_f=147$, $L_{ef}=33$, $L^{ef}=309$.}
\end{example}

\begin{example}
Let $i$ and $j$ be distinct elements of a $24$-element set $E$,
and let $\mathbb{V}$ be the set of blocks of the Steiner system $\mathrm{S}(5,8,24)$ that contain exactly one of $i$ and $j$.
Consider the collection of all $6$-element subsets of $E$ not contained in any member of $\mathbb{V}$. If we choose one such $\mathrm{B}$ uniformly at random,
we have
\begin{align*}
\mathbb{P}(i \in \mathrm{B}, j \in \mathrm{B}) \ \mathbb{P}(i \notin \mathrm{B}, j \notin \mathrm{B})&=\frac{7315}{124740}\cdot \frac{72149}{124740} \simeq 0.03391,\\
\mathbb{P}(i \in \mathrm{B}, j \notin \mathrm{B})\ \mathbb{P}(i \notin \mathrm{B}, j \in \mathrm{B})&=\frac{22638}{124740}\cdot \frac{22638}{124740} \simeq 0.03293.
\end{align*}
This example, due to Mark Jerrum, shows that a paving matroid need not have the negatively correlation property \cite[Section 4]{Jerrum}.
\end{example}

In Section \ref{sec:bound}, We use the Hodge theory for matroids in \cite{HW,AHK} to bound the correlation between the events $e \in \mathrm{B}$.

\begin{theorem}\label{thm:main1}
For  any distinct elements $i$ and $j$ in a matroid $\mathrm{M}$ of positive rank $d$,
\[
 \mathbb{P}(i \in \mathrm{B}, j \in \mathrm{B})\ \mathbb{P}(i \notin \mathrm{B}, j \notin \mathrm{B}) \le 
2 \Bigg(1-\frac{1}{d}\Bigg)\  \mathbb{P}(i \in \mathrm{B}, j \notin \mathrm{B})\ \mathbb{P}(i \notin \mathrm{B}, j \in \mathrm{B}).
\]
\end{theorem}

\noindent Theorem \ref{thm:main1} implies the covariance bound
\[
\text{Cov}(\text{$\mathrm{B}$ contains $i$}, \text{$\mathrm{B}$ contains $j$})  < \mathbb{P}(\text{$\mathrm{B}$ contains $i$})\ \mathbb{P}(\text{$\mathrm{B}$ contains $j$}).
\]
Compare the notion of \emph{approximate independence} in \cite[Section 4]{Kahn}.

An element $e$ of a rank $d$ matroid $\mathrm{M}$ is a \emph{loop} if it is contained in no basis of $\mathrm{M}$,  a \emph{coloop} if it is contained in every basis of $\mathrm{M}$, and 
\emph{free} if it is not a coloop and every circuit of $\mathrm{M}$ containing $e$ has cardinality $d+1$.
For example, the elements labelled $j$ in matroids of Examples \ref{ex:graphic} and \ref{ex:transversal} are free.
In Section \ref{sec:Mason}, we remove the factor $2$ in  Theorem \ref{thm:main1} when both $i$ and $j$ are free.

\begin{theorem}\label{thm:main2}
For any  distinct free elements $i$ and $j$ in a matroid $\mathrm{M}$ of positive rank $d$,
\[
 \mathbb{P}(i \in \mathrm{B}, j \in \mathrm{B})\ \mathbb{P}(i \notin \mathrm{B}, j \notin \mathrm{B}) \le 
 \Bigg(1-\frac{1}{d}\Bigg)\  \mathbb{P}(i \in \mathrm{B}, j \notin \mathrm{B})\ \mathbb{P}(i \notin \mathrm{B}, j \in \mathrm{B}).
\]
\end{theorem}

Can we replace the constant $2$ in Theorem \ref{thm:main1} by a smaller number? 
To any matroid $\mathrm{M}$, we associate a nonnegative real number $\alpha(\mathrm{M})$ defined by
\[
\alpha(\mathrm{M})=\sup \Big\{ \mathbb{P}(i \in \mathrm{B}, j \in \mathrm{B})\ \mathbb{P}(i \notin \mathrm{B}, j \notin \mathrm{B}) \hspace{0.5mm} / \hspace{0.5mm}  \mathbb{P}(i \in \mathrm{B}, j \notin \mathrm{B})\ \mathbb{P}(i \notin \mathrm{B}, j \in \mathrm{B}) \Big\},
\]
where the supremum is over all distinct non-loop non-coloop elements $i$ and  $j$ in $\mathrm{M}$ and all sets of positive weights $w$ on the elements of $\mathrm{M}$.
When every element of $\mathrm{M}$ is either a loop or a coloop, we set $\alpha(\mathrm{M})=0$.
It is straightforward to check that, if $\mathrm{M}^\perp$ is the dual matroid of $\mathrm{M}$ and $\mathrm{M}$ is a minor of another matroid $\mathrm{N}$, then
\[
\alpha(\mathrm{M}) = \alpha(\mathrm{M}^\perp) \ \ \text{and} \ \ \alpha(\mathrm{M}) \le \alpha(\mathrm{N}).
\]
In addition, if $\mathrm{M}_1$ and $\mathrm{M}_2$ have an element that is neither a loop nor a coloop, then
\[
\alpha(\mathrm{M}_1 \oplus \mathrm{M}_2) = \max\big\{\alpha(\mathrm{M}_1),\alpha(\mathrm{M}_2),1\big\}.
\]
We define the \emph{correlation constant}  $\alpha_\mathbb{F}$ of a field $\mathbb{F}$ to be the real number
\[
\alpha_\mathbb{F}=\sup_{} \Big\{ \alpha(\mathrm{M}) \Big\},
\]
where the supremum is over all matroids $\mathrm{M}$ representable over $\mathbb{F}$.
The \emph{correlation constant of matroids}, denoted $\alpha_{\text{Mat}}$, is defined in the same way by taking the supremum over all matroids.
As we can place any number of new elements in parallel to existing elements in any matroid, the values of $\alpha_\mathbb{F}$ and $\alpha_{\text{Mat}}$ remain unchanged if we only consider matroids with constant weights.

In Section \ref{sec:example}, We construct explicit examples to produce a lower bound of $\alpha_\mathbb{F}$ for any field $\mathbb{F}$.

\begin{theorem}\label{thm:main3}
The correlation constant of any field $\mathbb{F}$ satisfies
$\frac{8}{7} \le \alpha_\mathbb{F} \le \alpha_{\text{Mat}} \le 2$.
\end{theorem}

What is the correlation constant of $\mathbb{F}_2$?
What is the correlation constant of $\mathbb{C}$? Does $\alpha_\mathbb{F}$ depend on $\mathbb{F}$? 
What is the correlation constant $\alpha_{\text{Mat}}$?
The first question may be the most tractable one, as the only minor-minimal binary matroid with $\alpha(\mathrm{M})$ larger than $1$ is the matroid  represented over $\mathbb{F}_2$ by the matrix
\[
\left[\begin{array}{cccccccc}
1&0&0&0&0&1&1&1\\
0&1&0&0&1&0&1&1\\
0&0&1&0&1&1&0&1\\
0&0&0&1&1&1&1&1
\end{array}
\right].
\]
This matroid, labelled $\mathrm{S}_8$ in Oxley's list \cite[Appendix]{Oxley}, was first found by Seymour and Welsh to have positively correlated pair of elements \cite{SW}.
See \cite{CW} for a proof of the assertion on $\mathrm{S}_8$.
We conjecture, although without much evidence, that the correlation constant of $\mathbb{F}_2$ is $\frac{8}{7}$.
We know no matroid $\mathrm{M}$ with $\alpha(\mathrm{M})$ larger than $\frac{8}{7}$.

The initial motivation for our paper comes from the work of Mason \cite{Mason}, who offered the following three conjectures of increasing strength.
Several other authors studied correlations in matroid theory partly in pursuit of these conjectures \cite{SW,Wagner08,BBL,KN10,KN11}.

\begin{conjecture}\label{conj:Mason}
For any $n$-element matroid $\mathrm{N}$ and any positive integer $k$,
\begin{enumerate}[(1)]\itemsep 5pt
\item\label{MasonFirst} 
$
I_k(\mathrm{N})^2 \ge I_{k-1}(\mathrm{N})I_{k+1}(\mathrm{N}),
$
\item\label{MasonSecond} 
$
I_k(\mathrm{N})^2 \ge \frac{k+1}{k} I_{k-1}(\mathrm{N})I_{k+1}(\mathrm{N}),
$
\item\label{MasonThird} 
$
I_k(\mathrm{N})^2 \ge \frac{k+1}{k}\frac{n-k+1}{n-k} I_{k-1}(\mathrm{N})I_{k+1}(\mathrm{N}),
$
\end{enumerate}
where $I_k(\mathrm{N})$ is the number of $k$-element independent sets of $\mathrm{N}$.  
\end{conjecture}

Conjecture \ref{conj:Mason} (\ref{MasonFirst}) 
was proved 
 in \cite{AHK}.
Conjecture \ref{conj:Mason} (\ref{MasonThird}) is known to hold when 
$n$ is at most $11$ or $k$ is at most $5$  \cite{KN11}.  
We refer to \cite{Seymour,Dowling, Mahoney, Zhao, HK,HS, Lenz} for other partial results on Conjecture \ref{conj:Mason}.

Conjecture \ref{conj:Mason} (\ref{MasonSecond}) follows from the special case of Theorem \ref{thm:main2} when the weight $w$ is constant.

\begin{corollary}\label{cor:Mason}
Conjecture \ref{conj:Mason} (\ref{MasonSecond}) holds.
\end{corollary}

The implication is based on two standard constructions \cite[Chapter 7]{Oxley}. 
First, we use the truncation of $\mathrm{N}$ to reduce Conjecture \ref{conj:Mason} (\ref{MasonSecond})  to the case $k=d-1$,
where $d$ is the rank of $\mathrm{N}$.
Next, we construct the free extension $\mathrm{M}$ of  $\mathrm{N}$ by adding two new free elements $i$ and $j$.
If we pick a basis $\mathrm{B}$ of $\mathrm{M}$ uniformly at random, then
\begin{align*}
 \mathbb{P}(i \in \mathrm{B}, j \in \mathrm{B})\ \mathbb{P}(i \notin \mathrm{B}, j \notin \mathrm{B}) &= I_{d-2}(\mathrm{N}) \cdot I_{d}(\mathrm{N})\hspace{0.5mm} /\hspace{0.5mm} ( I_{d-2}(\mathrm{N}) + 2I_{d-1}(\mathrm{N}) + I_{d}(\mathrm{N}) )^2,\\
  \mathbb{P}(i \in \mathrm{B}, j \notin \mathrm{B})\ \mathbb{P}(i \notin \mathrm{B}, j \in \mathrm{B})  &= I_{d-1}(\mathrm{N}) \cdot I_{d-1}(\mathrm{N})\hspace{0.5mm} / \hspace{0.5mm}( I_{d-2}(\mathrm{N}) + 2I_{d-1}(\mathrm{N}) + I_{d}(\mathrm{N}) )^2.
\end{align*}
Now Conjecture \ref{conj:Mason} (\ref{MasonSecond}) for $\mathrm{N}$ is  Theorem \ref{thm:main2} for $i$ and $j$ in $\mathrm{M}$.

Conjecture \ref{conj:Mason} (\ref{MasonSecond}) implies an entropy bound that cannot be deduced from Conjecture \ref{conj:Mason} (\ref{MasonFirst}).
Recall that the \emph{Shannon entropy} $H(\mathrm{X})$ of a discrete random variable $\mathrm{X}$ is, by definition,
\[
H(\mathrm{X})=-\sum_{k}\mathbb{P}(\mathrm{X}=k) \log \mathbb{P}(\mathrm{X}=k),
\]
where the logarithm is in base $2$ and the sum is over all values of $\mathrm{X}$ with nonzero probability.
For a rank $d$ matroid $\mathrm{M}$, let $\mathrm{I}_\mathrm{M}$ be the size of an independent set drawn uniformly at random from the collection of all independent sets of $\mathrm{M}$.
For any $d$, uniform matroids of rank $d$ show that
\[
\inf_{\text{rk}(\mathrm{M})=d} H( \mathrm{I}_\mathrm{M} )=0,
\]
where the infimum is over all matroids of rank $d$.
We show that, asymptotically, the entropy of $\mathrm{I}_\mathrm{M}$ is at most half of the obvious upper bound $\log d$ given by Jensen's inequality.

\begin{corollary}\label{cor:entropy}
Uniform random independent sets of matroids satisfy
\[
\lim_{d \to \infty}\Bigg( \sup_{\text{rk}(\mathrm{M})=d} H( \mathrm{I}_\mathrm{M} ) / \log d \Bigg) =\frac{1}{2},
\]
where the supremum is over all matroids of rank $d$.
\end{corollary}

Corollary \ref{cor:entropy} is based on a result of Johnson \cite[Theorem 2.5]{Johnson}, who showed that the Poisson distribution maximizes entropy in the class of ultra log-concave distributions.
Recall that a random variable $\mathrm{X}$ taking its values in $\mathbb{N}$ is said to have the \emph{Poisson distribution with parameter $\lambda$} if
\[
\mathbb{P}(\mathrm{X}=k)=\frac{\lambda^k e^{-\lambda}}{k!} \ \  \text{for all $k \in \mathbb{N}$.}
\]
Combined with Conjecture \ref{conj:Mason} (\ref{MasonSecond}), Johnson's result implies that
\[
H(\mathrm{I}_\mathrm{M}) \le H(\mathrm{P}(\lambda)),
\]
where $\mathrm{P}(\lambda)$ is the Poisson distribution with parameter $\lambda=\mathbb{E}(\mathrm{I}_\mathrm{M})$.
Using known bounds for the entropy of Poisson distributions from information theory \cite[Theorem 8.6.5]{CT}, we get
\[
H(\mathrm{I}_\mathrm{M})  \le \frac{1}{2} \log \Bigg(2 \pi e\Big(d+\frac{1}{12}\Big)\Bigg).
\]
In general, an upper bound of the entropy of a random variable $\mathrm{X}$ implies a concentration of $\mathrm{X}$ \cite[Chapter 22]{Jukna}. The above bound of $H(\mathrm{I}_\mathrm{M})$, for example, gives the following.

\begin{corollary}\label{cor:concentration}
For any matroid $\mathrm{M}$ of rank $d$, there is $k$ such that
\[
\mathbb{P}(\mathrm{I}_\mathrm{M}=k) > \frac{1}{5\sqrt{d}}.
\]
\end{corollary}

\noindent Clearly, Corollaries \ref{cor:entropy} and \ref{cor:concentration} cannot be deduced from Conjecture \ref{conj:Mason} (\ref{MasonFirst}) alone.

\subsection*{Acknowledgements}

We thank Noga Alon, Petter Br\"and\'en, Jim Geelen,  Mark Jerrum, Matthew Kahle, Jaehoon Kim,  Russell Lyons,   Andrew Newman, and David Wagner
for valuable comments and discussions.

\section{Hodge theory for matroids}

We review the results of \cite{HW} and \cite{AHK} that will be used to prove Theorems  \ref{thm:main1} and \ref{thm:main2}.
For our purposes, we may assume that matroids do not have any loops.
In the rest of this paper, we fix a positive integer $n$ and  work with loopless matroids on finite sets
\[
E=\{1,\ldots,n\} \quad \text{and} \quad \overline{E}=\{0,1,\ldots,n\}.
\] 
Our notations will be consistent with those of \cite[Section 2]{HW}.

Let $\overline{\mathrm{M}}$ be a loopless matroid on $\overline{E}$ of rank $d+1$,
and let $\overline{\mathscr{L}}$ be the lattice of flats of $\overline{\mathrm{M}}$.
Introduce variables $x_{\overline{F}}$,
one for each nonempty proper flat $\overline{F}$ of $\overline{\mathrm{M}}$, and consider the polynomial ring
\[
S(\overline{\mathrm{M}})=\mathbb{R}[x_{\overline{F}}]_{\overline{F} \neq \varnothing, \overline{F} \neq \overline{E},\overline{F} \in \overline{\mathscr{L}}}.
\]
The \emph{Chow ring} $A(\overline{\mathrm{M}})$ is the quotient of $S(\overline{\mathrm{M}})$
by the ideal  generated by  the linear forms
\[
\sum_{e_1 \in \overline{F}} x_{\overline{F}} - \sum_{e_2 \in \overline{F}} x_{\overline{F}},
\]
one for each pair of distinct elements $e_1$ and $e_2$ of  $\overline{E}$, and the quadratic monomials
\[
x_{\overline{F}_1}x_{\overline{F}_2},
\]
one for each pair of incomparable nonempty proper flats $\overline{F}_1$ and $\overline{F}_2$ of $\overline{\mathrm{M}}$.
We denote the degree $q$ component of $A(\overline{\mathrm{M}})$ by $A^q(\overline{\mathrm{M}})$.

\begin{definition}
A real-valued function $c$ on $2^{\overline{E}}$  is said to be \emph{strictly submodular} if
$c_\varnothing=0$, $c_{\overline{E}}=0$,
and, for any pair of incomparable subsets $I_1,I_2 \subseteq \overline{E}$, we have
\[
 c_{I_1}+c_{I_2} > c_{I_1\,\cap\,I_2} +c_{I_1\,\cup\,I_2}.
 \]
  A strictly submodular function $c$ defines an element $\mathrm{L}(c)= \sum_{\overline{F}} c_{\overline{F}} x_{\overline{F}}$ in $A^1(\overline{\mathrm{M}})$.
\end{definition}

Note that strictly submodular functions on $2^{\overline{E}}$  exist. For example, we have the function
\[
c_I=(\text{number of elements in $I$})(\text{number of elements not in $I$}).
\]

We may now state the hard Lefschetz theorem and the Hodge-Riemann relations for matroids 
\cite[Theorem 1.4]{AHK}.
The function ``$\text{deg}$'' in Theorem \ref{HLHR} is the isomorphism  
$A^d(\overline{\mathrm{M}}) \simeq \mathbb{R}$
constructed in \cite[Section 5.3]{AHK}. 
This isomorphism is uniquely determined by its property 
\[
\text{deg}(x_{\overline{F}_1}x_{\overline{F}_2} \cdots x_{\overline{F}_d})=1 \ \ \text{for any chain of nonempty proper flats $\overline{F}_1 \subsetneq \overline{F}_2 \subsetneq \cdots \subsetneq \overline{F}_d$ in $\overline{\mathrm{M}}$.}
\]

\begin{theorem}\label{HLHR}
Let $\mathrm{L}$ be an element of $A^1(\overline{\mathrm{M}})$ attached to a strictly submodular function on $2^{\overline{E}}$.
\begin{enumerate}[(1)]\itemsep 5pt
\item (Hard Lefschetz theorem) For every nonnegative integer $q \le \frac{d}{2}$, the multiplication by $\mathrm{L}$ defines an isomorphism
\[
A^q(\overline{\mathrm{M}}) \longrightarrow A^{d-q}(\overline{\mathrm{M}}), \qquad \eta \longmapsto  \mathrm{L}^{d-2q} \ \eta.
\]
\item (Hodge-Riemann relations) For every nonnegative integer $q \le \frac{d}{2}$, the multiplication by $\mathrm{L}$ defines a symmetric bilinear form
\[
A^q(\overline{\mathrm{M}}) \times A^q(\overline{\mathrm{M}}) \longrightarrow \mathbb{R},  \qquad (\eta_1,\eta_2) \longmapsto (-1)^q \ \text{deg}(   \eta_1 \eta_2 \mathrm{L}^{d-2q})
\]
that is positive definite on the kernel of 
$\mathrm{L}^{d-2q+1}$.
\end{enumerate} 
\end{theorem}

Theorems  \ref{thm:main1} and \ref{thm:main2}, as well as other applications of the Hodge-Riemann relations in combinatorics surveyed in \cite{Huh},  only use the special case $q\le 1$.
 It will be interesting to find applications of the Hodge--Riemann relations for $q > 1$.

\section{Proof of Theorem \ref{thm:main1}}\label{sec:bound}

Let $\mathrm{M}$ be a rank $d$ loopless matroid on $E$. 
Let $\overline{\mathrm{M}}$ be the matroid on $\overline{E}$ obtained from $\mathrm{M}$ by adding $0$ as a coloop,
the direct sum of $\mathrm{M}$ and the rank $1$ matroid on $\{0\}$.
For every $e$ in $E$, we define an element 
\[
y_e=\sum_{0 \in \overline{F}, e \notin \overline{F}} x_{\overline{F}},
\]
where the sum is over all flats $\overline{F}$ of $\overline{\mathrm{M}}$ that contain $0$ and do not contain $e$.
The linear relations in $A(\overline{\mathrm{M}})$ show that we may equivalently define $y_e$ by summing over
all  flats $\overline{F}$ of $\overline{\mathrm{M}}$ that contain $e$ and do not contain $0$.
The quadratic relations in  $A(\overline{\mathrm{M}})$ show that, for any nonempty proper flat $\overline{F}$ of $\overline{\mathrm{M}}$ containing exactly one of $e$ and $0$,
\[
x_{\overline{F}} \cdot y_e=0.
\]
In what follows, relations of the above kind will be called \emph{$xy$-relations}.
The $xy$-relations imply that, for example,
 $y_e \cdot y_e$ is zero for any $e$ in $E$.

\begin{lemma}\label{lem:dependent}
For any dependent set $J$ of $\mathrm{M}$, we have
\[
\prod_{e \in J} y_e =0.
\]
\end{lemma}

\begin{proof}
We may suppose that $J$ is a circuit of $\mathrm{M}$.
Choose a maximal independent set $I$ of $\mathrm{M}$ in $J$,
 an element $f$ in $I$, and an element $g$ in $J \setminus I$.
Since $(I \setminus f) \cup g$ is a basis of $J$,
the set of flats of $\overline{\mathrm{M}}$ containing $(I \setminus f) \cup 0$ and not containing $f$ is equal to the set 
of flats  of $\overline{\mathrm{M}}$ containing $(I \setminus f) \cup 0$ and not containing $g$.
Therefore, by the $xy$-relations, we have
\[
\prod_{e \in I} y_e=y_{f}\prod_{e \in I \setminus f} y_e=y_{g}\prod_{e \in I \setminus f} y_e.
\]
Since the square of $y_g$ is zero, this gives $\prod_{e \in J} y_e=\prod_{e \in I} y_e\prod_{e \in J \setminus I} y_e =0$.
\end{proof}

\begin{lemma}\label{lem:basisdegree}
For any $d$-element subset $B$ of $E$, we have
\[
\text{deg}\Bigg(\prod_{e \in B}y_{e} \Bigg)=\left\{\begin{array}{ll} 
1&\text{if $B$ is a basis of $\mathrm{M}$,}\\
0&\text{if $B$ is not a basis of $\mathrm{M}$.}
\end{array}\right.
\]
\end{lemma}

\begin{proof}
Without loss of generality, we may suppose that $B=\{1,\ldots,d\}$. We consider
 the flats
\[
 \overline{F}_k=\text{the smallest flat of $\overline{\mathrm{M}}$ containing  $0,1,\ldots,k-1$},\ \ \text{for $k=1,\ldots,d+1$}.
\] 
If $B$ is a basis of $\mathrm{M}$, 
 $\overline{F}_{k}$ is the only flat of $\overline{\mathrm{M}}$ containing $0,1,\ldots,k-1$, not containing $k$, and  comparable to $\overline{F}_{k+1}$.
Thus the $xy$-relations imply that
\begin{align*}
y_{1}\cdots y_{{d-2}}y_{{d-1}}y_{d}&=\Big(y_{1}\cdots y_{{d-2}}y_{{d-1}}\Big)x_{\overline{F}_d}\\
&=\Big(y_{1} \cdots y_{{d-2}}\Big)x_{\overline{F}_{d-1}}x_{\overline{F}_d}=\cdots =x_{\overline{F}_1} \cdots x_{\overline{F}_{d-2}}x_{\overline{F}_{d-1}}x_{\overline{F}_d}.
\end{align*}
If $B$ is not a basis of $\mathrm{M}$, it contains a dependent set of $\mathrm{M}$, and  $\prod_{e \in B} y_e=0$ by Lemma \ref{lem:dependent}.
\end{proof}

\begin{lemma}\label{lem:c(e)}
Let $e$ be an element of $E$, and let $c(e)$ be the real-valued function on $2^{\overline{E}}$ defined by
\[
c(e)_I=\left\{\begin{array}{ll} 
1&\text{if $I$ contains $0$ and $I$ does not contain $e$,}\\
0&\text{if $I$ contains $e$ or $I$ does not contain $0$.}
\end{array}\right.
\]
Then $c(e)_\varnothing=0$, $c(e)_{\overline{E}}=0$, and, for any subsets $I_1,I_2$ of $\overline{E}$, we have
\[
c(e)_{I_1}+c(e)_{I_2} \ge c(e)_{I_1 \cap I_2} +c(e)_{I_1 \cup I_2}.
\]
\end{lemma}

The submodular inequality of Lemma \ref{lem:c(e)} is straightforward to check. In fact, we have
\[
c(e)_{I_1}+c(e)_{I_2} - c(e)_{I_1 \cap I_2} -c(e)_{I_1 \cup I_2}
=\left\{\begin{array}{ll} 
1&\text{if $0$ is in $I_1 \setminus I_2$ and $e$ is in $I_2 \setminus I_1$,}\\
1&\text{if $0$ is in $I_2 \setminus I_1$ and $e$ is in $I_1 \setminus I_2$,}\\
0&\text{if otherwise.}
\end{array}\right.
\]

We are ready to prove Theorem \ref{thm:main1}.
The equality holds in Theorem \ref{thm:main1} when $d=1$. 
Suppose from now on that $d \ge 2$.
Let $w=(w_e)$ be the given set of positive weights on $E$.
For distinct elements $i$ and $j$ in $E$, define
\[
\mathrm{L}_{ij}=\mathrm{L}_{ij}(w)=\sum_{e \neq i, e \neq j} w_ey_e,
\]
where the sum is over all elements of $E$ other than $i$ and $j$.
Lemma \ref{lem:basisdegree} shows that
\[
\text{deg} \big(\mathrm{L}_{ij}^d\big)=d! \Bigg( \sum_{B \in \mathcal{B}^{ij}} \prod_{e \in B} w_e\Bigg),
\]
where $\mathcal{B}^{ij}$ is the set of bases of $\mathrm{M}$ not containing $i$ and not containing $j$. Similarly,
\[
\text{deg} \big(y_i\mathrm{L}_{ij}^{d-1}\big)=(d-1)! \Bigg( \sum_{B \in \mathcal{B}^{j}_i} \prod_{e \in B} w_e\Bigg),
\]
where $\mathcal{B}^{j}_i$ is the set of bases of $\mathrm{M}$ containing $i$ and not containing $j$, and
\[
\text{deg} \big(y_iy_j\mathrm{L}_{ij}^{d-2}\big)=(d-2)! \Bigg( \sum_{B \in \mathcal{B}_{ij}} \prod_{e \in B} w_e\Bigg).
\]
where $\mathcal{B}_{ij}$ is the set of bases of $\mathrm{M}$ containing $i$ and  containing $j$.
Theorem \ref{thm:main1} obviously holds if $\mathcal{B}^{ij}$ or $\mathcal{B}_{ij}$ is empty. 
We suppose from now on that $\mathcal{B}^{ij}$ and $\mathcal{B}_{ij}$ are nonempty. 

Let $\mathrm{L}$ be any element of $A^1(\overline{\mathrm{M}})$ attached to a strictly submodular function on $2^{\overline{E}}$.
By Lemma \ref{lem:c(e)}, Theorem \ref{HLHR} applies to the element
$
\mathrm{L}_{ij}+\epsilon\mathrm{L}$  for any positive real number $\epsilon$.
By the Hodge-Riemann relations for $q\le 1$,
 any matrix representing the symmetric bilinear form
\[
A^1(\overline{\mathrm{M}})\times A^1(\overline{\mathrm{M}}) \longrightarrow \mathbb{R}, \qquad (\eta_1,\eta_2) \longmapsto \text{deg}\Big(\eta_1\eta_2\big(\mathrm{L}_{ij}+\epsilon\mathrm{L}\big)^{d-2}\Big)
\]
must have exactly one positive eigenvalue.
Thus, by continuity, any matrix representing the symmetric bilinear form
\[
A^1(\overline{\mathrm{M}})\times A^1(\overline{\mathrm{M}}) \longrightarrow \mathbb{R}, \qquad (a_1,a_2) \longmapsto \text{deg}\Big(\eta_1\eta_2 \mathrm{L}_{ij}^{d-2}\Big)
\]
has at most one positive eigenvalue.
Now consider the symmetric matrix
\[
\mathrm{H}_{ij}=\left[\begin{array}{ccc}
0 & \text{deg} \big(\hspace{0.3mm}y_i\hspace{0.3mm}y_j\hspace{0.3mm}\mathrm{L}_{ij}^{d-2}\big) & \text{deg} \big(\hspace{0.3mm}y_i\hspace{0.3mm}\mathrm{L}_{ij}\mathrm{L}_{ij}^{d-2}\big)\\
\text{deg} \big(\hspace{0.3mm}y_i\hspace{0.3mm}y_j\hspace{0.3mm}\mathrm{L}_{ij}^{d-2}\big) & 0 & \text{deg} \big(\hspace{0.3mm}y_j\hspace{0.3mm}\mathrm{L}_{ij}\mathrm{L}_{ij}^{d-2}\big)\\
\text{deg} \big(\hspace{0.3mm}y_i\hspace{0.3mm}\mathrm{L}_{ij}\mathrm{L}_{ij}^{d-2}\big) & \text{deg} \big(\hspace{0.3mm}y_j\hspace{0.3mm}\mathrm{L}_{ij}\mathrm{L}_{ij}^{d-2}\big) & \text{deg} \big(\mathrm{L}_{ij}\mathrm{L}_{ij}\mathrm{L}_{ij}^{d-2}\big)
\end{array}\right].
\]
Cauchy's eigenvalue interlacing theorem shows that $\mathrm{H}_{ij}$ has at most one positive eigenvalue as well.
On the other hand, $\mathrm{H}_{ij}$ has at least one positive eigenvalue, because its lower-right diagonal entry is positive.
A straightforward computation reveals that the determinant of $\mathrm{H}_{ij}$ is a positive multiple of
\[
2\Bigg(1-\frac{1}{d}\Bigg)\Bigg(\sum_{B \in \mathcal{B}_i^j} \prod_{e \in B^{\phantom{d}}} w_e\Bigg)\Bigg(\sum_{B \in \mathcal{B}_j^i} \prod_{e \in B} w_e\Bigg)-\Bigg(\sum_{B \in \mathcal{B}^{ij}} \prod_{e \in B} w_e\Bigg)\Bigg(\sum_{B \in \mathcal{B}_{ij}} \prod_{e \in B} w_e\Bigg).
\]
The determinant must be nonnegative by the condition on the eigenvalues of $\mathrm{H}_{ij}$, and hence
\[
 \mathbb{P}(i \in \mathrm{B}, j \in \mathrm{B})\ \mathbb{P}(i \notin \mathrm{B}, j \notin \mathrm{B}) \le 
2 \Bigg(1-\frac{1}{d}\Bigg)\  \mathbb{P}(i \in \mathrm{B}, j \notin \mathrm{B})\ \mathbb{P}(i \notin \mathrm{B}, j \in \mathrm{B}).
\]
This completes the proof of Theorem \ref{thm:main1}.

\section{Proof of Theorem \ref{thm:main2}}\label{sec:Mason}


Let $i$ and $j$ be distinct free elements in a rank $d$ matroid $\mathrm{Y}$,
and let $\mathrm{Z}$ be the deletion of $i$ and $j$ from $\mathrm{Y}$. 
We prove Theorem \ref{thm:main2} for $i$ and $j$ in $\mathrm{Y}$.
When $d=1$, no basis of $\mathrm{Y}$ contains both $i$ and $j$, and
 the equality holds in Theorem \ref{thm:main2}.
 Suppose from now on that $d \ge 2$.

Write $\mathcal{B}_{ij}$, $\mathcal{B}^i_{j}$, $\mathcal{B}^j_{i}$, $\mathcal{B}^{ij}$ for the set of bases  containing and/or not containing $i,j$,
and $\mathcal{I}_m$ for the collection of $m$-element independent sets.
Since $i$ and $j$ are free, we have natural bijections
\[
\mathcal{I}_{d}(\mathrm{Z}) \simeq \mathcal{B}^{ij}(\mathrm{Y}), \quad
\mathcal{I}_{d-1}(\mathrm{Z}) \simeq \mathcal{B}^{j}_i(\mathrm{Y}) \simeq \mathcal{B}^{i}_j(\mathrm{Y}), \quad
\mathcal{I}_{d-2}(\mathrm{Z}) \simeq \mathcal{B}_{ij}(\mathrm{Y}).
\]
If the rank of $\mathrm{Z}$ is less than $d$, Theorem \ref{thm:main2} clearly holds, as the left-hand side of the inequality is zero.
If the rank  of $\mathrm{Z}$ is $d$, Theorem \ref{thm:main2} follows from the following version of Corollary \ref{cor:Mason} applied to $\mathrm{Z}$.
 
 \begin{proposition}\label{prop:equivalent}
For any matroid $\mathrm{M}$ of rank $d \ge 2$ and any set of positive weights $w=(w_e)$,
\[
\Bigg(\sum_{I \in \mathcal{I}_{d-1}}\prod_{e \in I} w_e\Bigg)^2 \ge \frac{d}{d-1}\Bigg(\sum_{I \in \mathcal{I}_{d-2}}\prod_{e \in I} w_e\Bigg)\Bigg(\sum_{I \in \mathcal{I}_{d}}\prod_{e \in I} w_e\Bigg),
\]
where $\mathcal{I}_m=\mathcal{I}_m(\mathrm{M})$ is the collection of $m$-element independent sets of $\mathrm{M}$.
\end{proposition}

The proof of Proposition \ref{prop:equivalent} is similar to that of Theorem \ref{thm:main1}. 
We define an element 
\[
\alpha=\sum_{0 \in \overline{F}} x_{\overline{F}},
\]
where the sum is over all proper flats $\overline{F}$ of $\overline{\mathrm{M}}$ containing $0$.
The linear relations in  $A(\overline{\mathrm{M}})$ show that we may equivalently define $\alpha$ by summing over
all  flats $\overline{F}$ of $\overline{\mathrm{M}}$ containing $e$, for any $e$ in $E$.
The main ingredient of the proof is the following extension of Lemma \ref{lem:basisdegree}.

\begin{lemma}\label{lem:alpha}
For any $m$-element subset $I$ of $E$, we have
\[
\text{deg}\Bigg(\alpha^{d-m}\prod_{e \in I} y_e\Bigg)=\left\{\begin{array}{ll} 
1&\text{if $I$ is independent in $\mathrm{M}$,}\\
0&\text{if $I$ is dependent in $\mathrm{M}$.}
\end{array}\right.
\]
\end{lemma}

\begin{proof}
We prove by descending induction on $m$. The case $m=d$ is Lemma \ref{lem:basisdegree},
and the case of dependent $I$ is Lemma \ref{lem:dependent}.
For the induction step, suppose without loss of generality that $\{1,\ldots,d\}$ is a basis of $\mathrm{M}$.
It is enough to show that
\[
\Big(y_1\cdots y_{m-1}\Big) \ y_m \ \alpha^{d-m}=\Big(y_1\cdots y_{m-1}\Big)  \ \alpha^{d-m+1}.
\]
By the $xy$-relations, the difference of the right-hand side and the left-hand side is
\[
\Big(y_1\cdots y_{m-1}\Big) \ \Bigg(\sum_{\overline{G}} x_{\overline{G}}\Bigg) \ \alpha^{d-m},
\]
where the sum is over all proper flats $\overline{G}$ of $\overline{\mathrm{M}}$ containing $0,1,\ldots,m$.
For any such  $\overline{G}$, we claim  
\[
x_{\overline{G}} \ \alpha^{d-m}=0.
\]
To see this, use the linear relations in $A(\overline{\mathrm{M}})$ to write
\[
x_{\overline{G}} \ \alpha^{d-m}
=x_{\overline{G}} \ \Bigg(\sum_{\overline{F}_{m+1}} x_{\overline{F}_{m+1}}\Bigg) \ \ldots \ \Bigg(\sum_{\overline{F}_d} x_{\overline{F}_d}\Bigg),
\]
where the $k$-th sum is over all proper flats $\overline{F}_{m+k}$ of $\overline{\mathrm{M}}$ containing $m+k$.
Since $\{1,\ldots,d\}$ is a basis of $\mathrm{M}$, no proper flat of $\overline{\mathrm{M}}$ contains $\{0,1,\ldots,d\}$,
and hence the right-hand side  is zero by the quadratic relations in  $A(\overline{\mathrm{M}})$.
\end{proof}

We are ready to prove Proposition \ref{prop:equivalent}. Define another element 
\[
\mathrm{L}_0=\mathrm{L}_0(w)=\sum_{e \in E} w_ey_e, 
\]
where the sum is over all elements $e$ in $E$.
By Lemma \ref{lem:alpha}, for any nonnegative ineteger $m \le d$,
\[
\text{deg}\Big(\alpha^{d-m} \mathrm{L}_0^m\Big)=m!\Bigg(\sum_{I \in \mathcal{I}_m} \prod_{e \in I} w_e\Bigg),
\]
where $\mathcal{I}_m$ is the collection of $m$-element independent sets of $\mathrm{M}$.

Let $\mathrm{L}$ be any element of $A^1(\overline{\mathrm{M}})$ attached to a strictly submodular function on $2^{\overline{E}}$.
By Lemma \ref{lem:c(e)}, Theorem \ref{HLHR} applies to the element
$\mathrm{L}_{0}+\epsilon\mathrm{L}$  for any positive real number $\epsilon$.
By the Hodge-Riemann relations for $q\le 1$,
 any matrix representing the symmetric bilinear form
\[
A^1(\overline{\mathrm{M}})\times A^1(\overline{\mathrm{M}}) \longrightarrow \mathbb{R}, \qquad (\eta_1,\eta_2) \longmapsto \text{deg}\Big(\eta_1\eta_2\big(\mathrm{L}_{0}+\epsilon\mathrm{L}\big)^{d-2}\Big)
\]
must have exactly one positive eigenvalue.
Thus any matrix representing the symmetric bilinear form
\[
A^1(\overline{\mathrm{M}})\times A^1(\overline{\mathrm{M}}) \longrightarrow \mathbb{R}, \qquad (a_1,a_2) \longmapsto \text{deg}\Big(\eta_1\eta_2\mathrm{L}_{0}^{d-2}\Big)
\]
has at most one positive eigenvalue.
Now consider the symmetric matrix
\[
\mathrm{H}_{0}=\left[\begin{array}{cc}
\text{deg} \big(\hspace{0.3mm}\alpha\hspace{0.3mm}\alpha\hspace{0.5mm}\mathrm{L}_{0}^{d-2}\big)  & \text{deg} \big(\hspace{0.3mm}\alpha\hspace{0.5mm} \mathrm{L}_0\mathrm{L}_{0}^{d-2}\big)  \\
\text{deg} \big(\alpha \mathrm{L}_0 \mathrm{L}_{0}^{d-2}\big) & \text{deg} \big(\mathrm{L}_0 \mathrm{L}_0 \mathrm{L}_{0}^{d-2}\big) 
\end{array}\right].
\]
Cauchy's eigenvalue interlacing theorem shows that $\mathrm{H}_{0}$ has at most one positive eigenvalue.
On the other hand, $\mathrm{H}_{0}$ has at least one positive eigenvalue, because its lower-right diagonal entry is positive.
The determinant of $\mathrm{H}_{0}$ is a positive multiple of
\[
\frac{d}{d-1}\Bigg(\sum_{I \in \mathcal{I}_d} \prod_{e \in I} w_e\Bigg) \Bigg(\sum_{I \in \mathcal{I}_{d-2}} \prod_{e \in I} w_e\Bigg)
-\Bigg(\sum_{I \in \mathcal{I}_{d-1}} \prod_{e \in I} w_e\Bigg)^2,
\]
which must be nonpositive by the condition on the eigenvalues of $\mathrm{H}_{0}$.



\section{Proof of Theorem \ref{thm:main3}}\label{sec:example}

The upper bound follows from Theorem \ref{thm:main1}.
We construct explicit vector configurations over $\mathbb{F}$ to show the lower bound $\frac{8}{7}$.

Fix a prime number $p$ and an integer  $d \ge 2$.
Consider the $d$-dimensional vector space $\mathbb{F}_p^d$
over the field with $p$ elements, and let $\mathbf{e}_1,\mathbf{e}_2,\ldots,\mathbf{e}_d$ be the standard basis vectors of $\mathbb{F}_p^d$.

\begin{definition}\label{def:vectors}
Let $\mathrm{M}_{p}^d$ be the rank $d$ matroid represented by the vectors
$\mathbf{e}_1$, $\mathbf{e}_2+\cdots+\mathbf{e}_d$, and
\[
\begin{array}{cccc}
1\mathbf{e}_1+\mathbf{e}_2, &  2\mathbf{e}_1+\mathbf{e}_2, & \cdots &  p\mathbf{e}_1+\mathbf{e}_2,\\
1\mathbf{e}_1+\mathbf{e}_3, &  2\mathbf{e}_1+\mathbf{e}_3, & \cdots &  p\mathbf{e}_1+\mathbf{e}_3,\\
\vdots & \vdots &\ddots&\vdots \\
1\mathbf{e}_1+\mathbf{e}_d, &  2\mathbf{e}_1+\mathbf{e}_d, & \cdots & p\mathbf{e}_1+\mathbf{e}_d.
\end{array}
\]
We write $i$ for the vector $\mathbf{e}_1$ and $j$ for the vector $\mathbf{e}_2+\cdots+\mathbf{e}_d$. 
\end{definition}

The matroid $\mathrm{M}^4_2$ is isomorphic to the  matroid $\mathrm{S}_8$ mentioned in the introduction.
For any $d$, the matroid $\mathrm{M}^d_2$ is the self-dual matroid obtained from the binary spike $\mathrm{Z}_d$ in Oxley's list  by deleting any element other than the tip \cite[Appendix]{Oxley}.\footnote{According to Geelen \cite{Geelen}, ``it all goes wrong for spikes.'' The spike $\mathrm{Z}_d$ was first used by Seymour to demonstrate that
an independence oracle algorithm for testing whether a matroid is binary cannot run in polynomial time relative to the size of the ground set \cite{Seymour81}.} For any $p$, the matroid $\mathrm{M}_p^d$ has a spike-like structure in that it has a ``tip'' $i$ and ``legs''
\[
\mathrm{L}_m=\Big\{ \mathbf{e}_1, 1\mathbf{e}_1+\mathbf{e}_m, 2\mathbf{e}_1+\mathbf{e}_m,\ldots, p\mathbf{e}_1+\mathbf{e}_m\Big\} \ \ \text{for $m=2,\ldots,d$}.
\]
For general spikes and their role in structural matroid theory, see \cite[Chapter 14]{Oxley}.
As before, we write $\mathcal{B}_{ij}$, $\mathcal{B}^i_{j}$, $\mathcal{B}^j_{i}$, $\mathcal{B}^{ij}$ for the set of bases  of $\mathrm{M}=\mathrm{M}^d_p$ containing and/or not containing $i,j$.

\begin{enumerate}[(1)]\itemsep 5pt
\item The contraction $\mathrm{M} / i / j$ is  
the uniform matroid $\mathrm{U}_{d-2,d-1}$ with each element replaced by $p$ parallel copies. 
Any basis of the contraction is disjoint from one of the parallel classes and contains exactly one point from each one of the remaining parallel classes.
Therefore, 
\[
|\mathcal{B}_{ij}(\mathrm{M})|=(d-1){p\choose 1}^{d-2}.
\]
\item The deletion $\mathrm{M} \setminus i \setminus j$ is represented by the $p$-point lines $\mathrm{L}_2 \setminus \mathbf{e}_1$, $\ldots$, $\mathrm{L}_d \setminus \mathbf{e}_1$ in $\mathbb{F}_p^d$.
Any basis of the deletion must contain exactly two points from one of the lines and one point from each one of the remaining lines. 
Therefore, 
\[
|\mathcal{B}^{ij}(\mathrm{M})|=(d-1) {p \choose 2} {p \choose 1}^{d-2}.
\]
\item The contraction-deletion $\mathrm{M} /i \setminus  j$ is the boolean matroid $\mathrm{U}_{d-1,d-1}$ with each element replaced by $p$ parallel copies. 
Any basis of the contraction-deletion contains exactly one element from each parallel class.
Therefore,
\[
|\mathcal{B}_{i}^j(\mathrm{M})|={p\choose 1}^{d-1}.
\]
\end{enumerate}

It remains to compute the number of bases of $\mathrm{M}$ not containing $i$ and containing $j$.
There are two types of such bases, corresponding to the two terms in the right-hand side of
\[
|\mathcal{B}^i_j(\mathrm{M})|=(p^{d-1}-p^{d-2})+(d-1)(d-2){p \choose 2}{p \choose 1}^{d-3}.
\]
A basis of the first type contains exactly one point from each one of the $p$-point lines $\mathrm{L}_2 \setminus \mathbf{e}_1$, $\ldots$, $\mathrm{L}_d \setminus \mathbf{e}_1$.
The determinant formula
\[
\det\left[\begin{array}{cccccc}
0&k_2&k_3&\cdots&k_d\\
1&1&0&\cdots&0\\
1&0&1&\cdots&0\\
\vdots&\vdots&\vdots&\ddots&\vdots\\
1&0&0&\cdots&1
\end{array}\right]=-k_2-k_3-\cdots-k_d
\]
shows that there are exactly $(p^{d-1}-p^{d-2})$ such bases.
A basis of the second type contains exactly two points from one of the lines, 
no point from another, and one point from each one of the remaining lines.
It is clear that any basis in $\mathcal{B}^i_j$ must be one of the two types.

Combining  the four numbers, we obtain a ratio that depends only on $d$ and not on $p$:
\[
\frac{|\mathcal{B}_{ij}(\mathrm{M})||\mathcal{B}^{ij}(\mathrm{M})|}{|\hspace{0.5mm}\mathcal{B}^i_j(\mathrm{M})\hspace{0.5mm}| |\hspace{0.5mm}\mathcal{B}^i_j(\mathrm{M})\hspace{0.5mm}|}= \frac{d^2-2d+1}{d^2-3d+4}.
\]
The maximum of the ratio is $\frac{8}{7}$, achieved when $d=5$. This proves Theorem \ref{thm:main3} when the field $\mathbb{F}$ has characteristic $p$.

For fields of characteristic zero, 
let $i$ and $j$ be distinct elements of a finite set $A_1$.
Let $A_2,\ldots,A_d$ be a family of $(m+1)$-element subsets  of $A_1\setminus i$ whose union is $A_1 \setminus i$ and whose pairwise intersection is $\{j\}$.
We extend the transversal matroid construction in Example \ref{ex:transversal} as follows.

\begin{definition}
The matroid $\mathrm{N}_m^d$ is the transversal matroid of the family $A_1,A_2,\ldots,A_d$.
\end{definition}

The matroid $\mathrm{N}^6_2$ is isomorphic to the truncated graphic matroid in Example \ref{ex:graphic}.
By definition,  bases of $\mathrm{N}=\mathrm{N}_m^d$ are the systems of distinct representatives of $A_1,A_2,\ldots,A_d$.
For $m=p$, the matroids $\mathrm{M}_p^d$ and $\mathrm{N}_m^d$ share three of the four minors obtained by deleting and/or contracting $i,j$.
For any $m$, we have
\[
|\mathcal{B}_{ij}(\mathrm{N})|=(d-1){m\choose 1}^{d-2}, \quad |\mathcal{B}^{ij}(\mathrm{N})|=(d-1) {m \choose 2} {m \choose 1}^{d-2}, \quad |\mathcal{B}_{i}^j(\mathrm{N})|={m\choose 1}^{d-1}.
\]
There are two types of bases  of $\mathrm{N}$ not containing $i$ and containing $j$, corresponding to the two terms in the right-hand side of
\[
|\mathcal{B}^i_j(\mathrm{N})|=m^{d-1}+(d-1)(d-2){m \choose 2}{m \choose 1}^{d-3}.
\]
A basis of the first type contains exactly one element from each one of the sets $A_2 \setminus j,\ldots,A_d \setminus j$.
A basis of the second type contains exactly two points from one of the sets $A_k \setminus j$, 
no point from another $A_k \setminus j$, and one point from each one of the remaining $A_k \setminus j$.

Combining  the four numbers and taking the limit $m \to \infty$, we obtain the same ratio as before:
\[
\lim_{m \to \infty}\frac{|\mathcal{B}_{ij}(\mathrm{N})||\mathcal{B}^{ij}(\mathrm{N})|}{|\hspace{0.5mm}\mathcal{B}^i_j(\mathrm{N})\hspace{0.5mm}| |\hspace{0.5mm}\mathcal{B}^i_j(\mathrm{N})\hspace{0.5mm}|}
= \frac{d^2-2d+1}{d^2-3d+4}.
\]
Since transversal matroids are representable over any infinite field \cite[Chapter 11]{Oxley}, this proves Theorem \ref{thm:main3} when  $\mathbb{F}$ has characteristic $0$.
In fact, for any positive integer $p=m$, the set of vectors in Definition \ref{def:vectors}  viewed as elements of $\mathbb{Q}^d$ represents $\mathrm{N}^d_m$.

\section{Proofs of Corollaries \ref{cor:entropy} and \ref{cor:concentration}}\label{sec:corollaries}

Let $\mathrm{I}_\mathrm{M}$ be the size of an independent set drawn uniformly at random from the collection of all independent sets of a rank $d$ matroid $\mathrm{M}$.
As discussed in the introduction, Corollary \ref{cor:Mason} and \cite[Theorem 2.5]{Johnson} together imply
\[
H(\mathrm{I}_\mathrm{M})  \le \frac{1}{2} \log \Bigg(2 \pi e\Big(\mathbb{E}(I_\mathrm{M})+\frac{1}{12}\Big)\Bigg) \le \frac{1}{2} \log \Bigg(2 \pi e\Big(d+\frac{1}{12}\Big)\Bigg).
\]
Corollary \ref{cor:concentration} follows from the upper bound of $H(\mathrm{I}_\mathrm{M}) $ and the easy implication
\[
H(\mathrm{X}) \le \log t \Longrightarrow \max_k \mathbb{P}(X=k) \ge \frac{1}{t}.
\]
Corollary \ref{cor:entropy} follows from the upper bound of $H(\mathrm{I}_\mathrm{M}) $ and the estimate
\[
\frac{1}{2} \log\Bigg( \frac{\pi}{2} d \Bigg)\le\log \frac{2^d}{{d \choose d/2}} \le  \sum_{k=0}^d \frac{{d \choose k}}{2^d} \log\frac{2^d}{{d \choose k}} \le \sup_{\text{rk}(\mathrm{M})=d}H(\mathrm{I}_\mathrm{M}).
\]
The first inequality follows from Stirling's approximation,
the second inequality follows from  ${d \choose k} \le {d \choose d/2}$,
and the third inequality is witnessed
by the rank $d$ boolean matroid.


\begin{thebibliography}{MNWW11}


\bibitem[AHK18]{AHK} Karim Adiprasito, June Huh, and Eric Katz,
			\emph{Hodge theory for combinatorial geometries}.
			Ann. of Math. (2) {\bf 188} (2018), no. 2, to appear.
						


\bibitem[BBL09]{BBL} Julius Borcea, Petter Br\"and\'en, and Thomas Liggett, 
			\emph{Negative dependence and the geometry of polynomials}. 
			J. Amer. Math. Soc. {\bf 22} (2009), no. 2, 521--567. 




\bibitem[COSW04]{COSW} Youngbin Choe, James Oxley, Alan Sokal, and David Wagner, 
			\emph{Homogeneous multivariate polynomials with the half-plane property}.
			Special issue on the Tutte polynomial. 
			Adv. in Appl. Math. {\bf 32} (2004), no. 1-2, 88--187. 

\bibitem[CW06]{CW}  Youngbin Choe and David Wagner, 
			\emph{Rayleigh matroids}.
			Combin. Probab. Comput. {\bf 15} (2006), no. 5, 765--781. 


\bibitem[CT06]{CT} Thomas Cover and Joy Thomas, 
			\emph{Elements of information theory}.
			Second edition. Wiley-Interscience, Hoboken, NJ, 2006.

			
\bibitem[Dow80]{Dowling} Thomas Dowling,
			\emph{On the independent set numbers of a finite matroid}. 
			Combinatorics 79 (Proc. Colloq., Univ. Montr\'al, Montreal, Que., 1979), Part I. 
			Ann. Discrete Math. {\bf 8} (1980), 21--28. 

\bibitem[FM92]{FM} Tom\'as Feder and Milena Mihail,
			\emph{Balanced matroids},
			Proceedings of the 24th Annual ACM Symposium on Theory of Computing, 26--38, ACM Press, 1992.

\bibitem[Gee08]{Geelen} Jim Geelen, 
			\emph{Some open problems on excluding a uniform matroid}.
			Adv. in Appl. Math. {\bf 41} (2008), no. 4, 628--637.

			
			
\bibitem[HS89]{HS} Yahya Ould Hamidoune and Isabelle Sala\"un,
			\emph{On the independence numbers of a matroid}.
			J. Combin. Theory Ser. B {\bf 47} (1989), no. 2, 146--152. 

\bibitem[HK12]{HK} June Huh and Eric Katz,
			 \emph{Log-concavity of characteristic polynomials and the Bergman fan of matroids}. 
			 Math. Ann. {\bf 354} (2012), 1103--1116.

\bibitem[HW17]{HW} June Huh and Botong Wang,
			\emph{Enumeration of points, lines, planes, etc.}
			Acta Math. {\bf 218} (2017), no. 2, 297--317.

\bibitem[Huh18]{Huh} June Huh,
			\emph{Combinatorial applications of the Hodge-Riemann relations}.
			Proceedings of the International Congress of Mathematicians, 2018.			

\bibitem[Jer06]{Jerrum} Mark Jerrum,
			\emph{Two remarks concerning balanced matroids}. 
			Combinatorica {\bf 26} (2006), no. 6, 733--742. 
			
\bibitem[Joh07]{Johnson} Oliver Johnson,
			\emph{Log-concavity and the maximum entropy property of the Poisson distribution}.
			Stochastic Process. Appl. {\bf 117} (2007), no. 6, 791--802. 

\bibitem[Juk11]{Jukna} Stasys Jukna, 
			\emph{Extremal combinatorics, with applications in computer science}. Second edition. Texts in Theoretical Computer Science. Springer, Heidelberg, 2011.

\bibitem[Kah00]{Kahn} Jeff Kahn,
			\emph{A normal law for matchings}.
			Combinatorica {\bf 20} (2000), no. 3, 339--391. 
			
\bibitem[KN10]{KN10} Jeff Kahn and Michael Neiman,
			\emph{Negative correlation and log-concavity}.
			Random Structures Algorithms {\bf 37} (2010), no. 3, 367--388. 
			
\bibitem[KN11]{KN11} Jeff Kahn and Michael Neiman,
			\emph{A strong log-concavity property for measures on Boolean algebras}.
			J. Combin. Theory Ser. A {\bf 118} (2011), no. 6, 1749--1760. 



\bibitem[Len13]{Lenz}  Matthias Lenz,
			\emph{The f-vector of a representable-matroid complex is log-concave}. 
			Adv. in Appl. Math. {\bf 51} (2013), no. 5, 543--545.

\bibitem[LP16]{LP} Russell Lyons and Yuval Peres, 
			\emph{Probability on trees and networks}. 
			Cambridge Series in Statistical and Probabilistic Mathematics {\bf 42}, Cambridge University Press, New York, 2016.

			
\bibitem[Mah85]{Mahoney} Carolyn Mahoney,
			\emph{On the unimodality of the independent set numbers of a class of matroids},
			J. Combin. Theory Ser. B {\bf 39} (1985), no. 1, 77--85. 


\bibitem[Mas72]{Mason} John Mason, 
			\emph{Matroids: unimodal conjectures and Motzkin's theorem}. 
			Combinatorics (Proc. Conf. Combinatorial Math., Math. Inst., Oxford, 1972), 207--220, Inst. Math. Appl., Southend-on-Sea, 1972.



\bibitem[Oxl11]{Oxley} James Oxley, 
			\emph{Matroid theory}. 
			Second edition. Oxford Graduate Texts in Mathematics {\bf 21}. Oxford University Press, Oxford, 2011. 


\bibitem[Pem95]{Pemantle95} Robin Pemantle, 
			\emph{Uniform random spanning trees}. 
			Topics in contemporary probability and its applications, 1--54, 
			Probab. Stochastics Ser., CRC, Boca Raton, FL, 1995. 
						

			
\bibitem[Sey75]{Seymour} Paul Seymour,
			\emph{Matroids, hypergraphs, and the max-flow min-cut theorem}.
			Thesis, University of Oxford, 1975.

\bibitem[Sey81]{Seymour81} Paul Seymour, 
			\emph{Recognizing graphic matroids}. 
			Combinatorica {\bf 1} (1981), no. 1, 75--78.
			
\bibitem[SW75]{SW} Paul Seymour and Dominic Welsh,
			\emph{Combinatorial applications of an inequality from statistical mechanics}.
			Math. Proc. Cambridge Philos. Soc. {\bf 77} (1975), 485--495. 




\bibitem[Wag05]{Wagner05} David Wagner,
			\emph{Rank-three matroids are Rayleigh}.
			Electron. J. Combin. {\bf 12} (2005), Note 8, 11 pp. 
			
\bibitem[Wag08]{Wagner08} David Wagner,	
			\emph{Negatively correlated random variables and Mason's conjecture for independent sets in matroids}. 
			Ann. Comb. {\bf 12} (2008), no. 2, 211--239. 
			

\bibitem[Zha85]{Zhao} Cui Kui Zhao, 
			\emph{A conjecture on matroids}. 
			Neimenggu Daxue Xuebao {\bf 16} (1985), no. 3, 321--326. 			
					
\end{thebibliography}
\end{document}